\documentclass[11pt]{article}
\usepackage{amsmath, amsthm, amssymb, amsfonts, mathrsfs}
\usepackage[dvips]{graphicx}
\newtheorem{theorem}{Theorem}[section]
\newtheorem{cor}[theorem]{Corollary}

\newtheorem{lem}[theorem]{Lemma}
\newtheorem{prop}[theorem]{Proposition}

\newtheorem{ex}{Example}[section]

\def \Zl {{\mathbb Z}}
\def \Nl {{\mathbb N}}
\def \Rl {{\mathbb R}}

\def \Cl {{\mathbb C}}

\def \ul {{\mathbf u}}
\def \vl {{\mathbf v}}
\def \x {{\mathbf x}}
\def \y {{\mathbf y}}
\def \sig {{\mathbf \sigma}}

\def \o {{\mathbf 0}}

\def \m {{\mathbf m}}
\def \n {{\mathbf n}}
\def \D {{\mathbf D}}
\def \C {{\mathbf C}}
\def \d {{\mathbf d}}
\title{ Perfect State Transfer on gcd-graphs}

\author{ Hiranmoy Pal\\
Department of Mathematics\\
Indian Institute of Technology Guwahati\\Guwahati, India - 781039\\
Email: hiranmoy@iitg.ernet.in\\
\\
Bikash Bhattacharjya\\
Department of Mathematics\\
Indian Institute of Technology Guwahati\\Guwahati, India - 781039\\
Email: b.bikash@iitg.ernet.in
}
\begin{document}
\maketitle

\vspace{-0.3in}

\begin{center}{Abstract}\end{center}
Let $G$ be a graph with adjacency matrix $A$. The transition matrix of $G$ is denoted by $H(t)$ and it is defined by $H(t):=\exp{\left(itA\right)},\;t\in\Rl.$ The graph $G$ has perfect state transfer (PST) from a vertex $u$ to another vertex $v$ if there exist $\tau\left(\neq0\right)\in\Rl$ such that the $uv$-th entry of $H(\tau)$ has unit modulus. In case when $u=v$, we say that $G$ is periodic at the vertex $u$ at time $\tau$. The graph $G$ is said to be periodic if it is periodic at all vertices at the same time. A gcd-graph is a Cayley graph over a finite abelian group defined by greatest common divisors. We establish a sufficient condition for a gcd-graph to have periodicity and PST at $\frac{\pi}{2}$. Using this we deduce that there exists gcd-graph having PST over an abelian group of order divisible by $4$. Also we find a necessary and sufficient condition for a class of gcd-graphs to be periodic at $\pi$. Using this we characterize a class of gcd-graphs not exhibiting PST at $\frac{\pi}{2^{k}}$ for all positive integers $k$.


\section{Introduction} 
Perfect state transfer has a great significance in continuous-time quantum walks as it has applications in quantum information processing. Perfect sate transfer in quantum networks was introduced by Bose \cite{bose}. Quantum networks are modelled on finite graphs and when there is no external dynamic control over a system then PST depends only on the underlying graph. Main interest here is to find graphs exhibiting PST.\par
Christandl \emph{et al}. \cite{chr1,chr2} found that the paths $P_2$ and $P_3$ and their cartesian powers exhibits PST. A characterization of PST in NEPS (non-complete extended $p$-sum) of the path $P_3$ was given by Pal \emph{et al}. \cite{pal}. Bernasconi \emph{et al}. \cite{ber} showed that PST occurs on certain cubelike graphs which are actually gcd-graphs over a direct product of finitely many copies of the group $\Zl_2$. A complete characterization of PST in integral circulant networks was found by Ba\v si\'c \cite{mil,mil11}. The integral circulant graphs are also gcd-graphs over the cyclic group $\Zl_n$.\par
In this article we consider PST on gcd-graphs over general abelian groups. These graphs are known to have integral spectrum. It is well known that periodicity is a necessary condition for PST in regular graphs. In \cite{god1,god2}, Godsil found that a regular graph is periodic if and only if its eigenvalues are integers. Therefore, if a Cayley graph has PST then it must have integral spectrum. A characterization of integral Cayley graph over finite abelian groups is given in \cite{wal1}. Among the integral Cayley graphs we only consider gcd-graphs.  Here we find some characterizations of periodicity and PST in gcd-graphs.  More information regarding PST and periodicity can be found in \cite{chi,cou,ge,saxe,stev}.
\section{Preliminaries}
We introduce gcd-graphs over a finite abelian group. Let $\left(\Gamma,+\right)$ be a finite abelian group and consider $S\subseteq\Gamma$ with $\left\lbrace -s:s\in S\right\rbrace=S$. Such a set $S$ is called a symmetric subset of $\Gamma$. The Cayley graph over $\Gamma$ with connection set $S$ is denoted by $Cay\left(\Gamma,S\right)$. The graph has the vertex set $\Gamma$ where two vertices $a,b\in\Gamma$ are adjacent if and only if $a-b\in S$. If the additive identity $0\in S$ then $Cay\left(\Gamma,S\right)$ has loops at each of its vertices. In that case, we use the convention that \textbf{each loop contributes one} to the corresponding diagonal entry of the adjacency matrix. However while discussing PST on gcd-graphs, we consider simple graphs, \emph{i.e,} we consider $0\not\in S$.
\par The greatest common divisor of two non-negative integers $m,n$ is denoted by $gcd(m,n)$. For every non-negative integer $n$ we use the convention that $gcd(0,n)=gcd(n,0)=n$. Let us consider two $r$-tuples of non-negative integers $\m=\left(m_1,\ldots, m_r\right)$ and $\n=\left(n_1,\ldots, n_r\right)$. For $i=1,\ldots,r$, suppose $gcd(m_i,n_i)=d_i$ and we write $\d=\left(d_1,\ldots,d_r\right)$. We define gcd of $\m,\n$ to be $\d$ and write $gcd(\m,\n)=\d$.\par
Let $\Zl_n$ be the cyclic group of order $n$. Every finite abelian group $\left(\Gamma,+\right)$ has a cyclic group decomposition
$$\Gamma=\Zl_{m_1}\oplus\ldots\oplus\Zl_{m_r},\text{ where }r\geq 1\text{ and }m_i\geq 1 \text{ for } i=1,\ldots,r.$$
For each $i=1,\ldots, r$, assume that $d_i$ is a divisor of $m_i$ with $1\leq d_i\leq m_i$. For the divisor tuple $\d=\left(d_1,\ldots, d_r\right)$ of $\m=\left(m_1,\ldots, m_r\right)$, define \[S_\Gamma(\d)=\left\lbrace\x\in\Gamma : gcd(\x,\m)=\d\right\rbrace.\] Let $\D$ be a set of divisor tuples of $\m$ and define $$S_\Gamma(\D)=\bigcup\limits_{\d\in \D}S_\Gamma(\d).$$
Note that the union is actually a disjoint union. The sets $S_\Gamma(\D)$ are called gcd-sets of $\Gamma$. A Cayley graph over a finite abelian group whose connection set is a gcd-set is called a gcd-graph. More information regarding gcd-graphs can be found in \cite{wal,wal1}.\par
Consider two graphs $G_1$ and $G_2$ with the set of vertices $U$ and $V$, respectively. The Kronecker product \cite{cev} of $G_1$ and $G_2$ is denoted by $G_1\times G_2$ which has the vertex set $U\times V$. Two vertices $(u_{1},u_{2})$ and $(v_{1},v_{2})$ are adjacent in $G_1\times G_2$ whenever $u_{1}$ is adjacent to $v_{1}$ in $G_1$ and $u_{2}$ is adjacent to $v_{2}$ in $G_2$. If $G_1$ and $G_2$ have the adjacency matrices $A_1$ and $A_2$, respectively, then $G_1\times G_2$ has the adjacency matrix $A_1\otimes A_2$, the tensor product of $A_1$ and $A_2$. The transition matrix of Kronecker product of two graphs can be calculated as follows.
\begin{prop}\cite{god1,pal}\label{aa}
Let $G_1$ and $G_2$ be two graphs with adjacency matrices $A_1$ and $A_2$, respectively. Suppose the transition matrix of $G_1$ is $H(t)$. If the spectral decomposition of $A_2$ is $\sum\limits_{s=1}^{q}\mu_{s}F_{s}$ then $G_1\times G_2$ has the transition matrix $\sum\limits_{s=1}^{q} H(\mu_{s}t)\otimes F_{s}$.
\end{prop}
More information regarding PST on Kronecker products can be found in \cite{ god1}.

\section{PST on cubelike graphs}\label{cube}
A cubelike graph $X(\C)$ is a Cayley graph over $\Zl_{2}^{n}$ with a connection set $\C\subset\Zl_{2}^{n}$. PST on cubelike graphs (simple) has already been discussed in \cite{ber,chi}. Here we discuss some relevant results from \cite{chi}, which remain valid for looped cubelike graphs. Recall that each loop contributes one to the adjacency matrix. Finally we deduce a simple result that will be used later to characterize PST in gcd-graphs over general abelian groups.
\par For each $\x\in\Zl_{2}^{n}$, the map $P_{\x}:\Zl_{2}^{n}\longrightarrow\Zl_{2}^{n}$ defined by $P_{\x}\left(\y\right)=\x+\y$ is a permutation of the elements of $\Zl_{2}^{n}$ and hence it can be realized as a permutation matrix of appropriate order. It is easy to see that $P_{\o}=I$ and  $P_{\x}P_{\y}=P_{\x+\y}$ which imply $P_{\x}^{2}=I$. The following result finds the adjacency matrix of a cubelike graph.
\begin{lem}\label{2a}\cite{chi}
If $\C\subseteq\Zl_{2}^{n}$ and $X(\C)$ is the cubelike graph with connection set $\C$ then $X(\C)$ has the adjacency matrix $A=\sum\limits_{\x\in \C}P_{\x}.$
\end{lem}
Note that $P_{\x}P_{\y}=P_{\x+\y}=P_{\y}P_{\x}$ for all $\x,\y\in\Zl_n$. Therefore by using the property of matrix exponential we find that
$$\exp{\left(it(P_{\x}+P_{\y})\right)}=\exp{(itP_{\x})}\exp{(itP_{\y})}.$$ Using this the transition matrix of a cubelike graph can be calculated as follows.
\begin{lem}\label{2b}\cite{chi}
If $H(t)$ is the transition matrix of the cubelike graph $X(\C)$ then $$H(t)=\prod\limits_{\x\in \C}\exp{(itP_{\x})}.$$
\end{lem}
We already have $P_{\x}^{2}=I$ and this implies that
\begin{eqnarray*}
\exp{(itP_{\x})} &=& I+itP_{\x}-\frac{t^{2}}{2!}I-i\frac{t^{3}}{3!}P_{\x}+\frac{t^{4}}{4!}I+\ldots \\
&=&\cos{(t)}I+i\sin{(t)}P_{\x}.
\end{eqnarray*}
Observe that if $\sig$ is the sum of the elements in $\C\subseteq\Zl_{2}^{n}$ then $\prod\limits_{\x\in \C}P_{\x}=P_{\sig}$. Therefore by using Lemma \ref{2b}, we deduce that
$$H\left(\frac{\pi}{2}\right)=\prod\limits_{\x\in \C}iP_{\x}=i^{|\C|}P_{\sig}.$$ The next result determines periodicity and PST in cubelike graphs at $t=\frac{\pi}{2}$.
\begin{theorem}\label{2c}\cite{chi}
Let $\C\subseteq\Zl_{2}^{n}$ and let $\sig$ be the sum of the elements of $\C$. If $\sig\neq \o$ then PST occurs in $X(\C)$ from $\x$ to $\x+\sig$ at $\frac{\pi}{2}$. If $\sig=\o$ then $X(\C)$ is periodic with period $\frac{\pi}{2}$. 
\end{theorem}

Now we find a sufficient condition for periodicity in a class of graphs constructed from cubelike graphs. Consider the following result.

\begin{prop}\label{2d}
Let $\C\subseteq\Zl_{2}^{n}$ and let $X(\C)$ be a cubelike graph with the connection set $\C$. Assume that the sum of the elements of $\C$ is $\o$ and $|\C|\equiv 0\;(\text{mod}\;4)$. Then for every integral graph $G$, the transition matrix of $X(\C)\times G$ at $\frac{\pi}{2}$ is the identity operator.
\end{prop}
\begin{proof}
If $H(t)$ is the transition matrix of the cubelike graph $X(\C)$ then we have
$$H\left(\frac{\pi}{2}\right)=\prod\limits_{\x\in \C}iP_{\x}=i^{|\C|}P_{\sig},$$ where $\sig$ is the sum of the elements in $\C$. Now $\sig=\o$ implies that $P_{\sig}=I$ and therefore if $|\C|\equiv 0\;(\text{mod}\;4)$, then $H\left(\frac{\pi}{2}\right)=I$. For every integer $\mu$, we find that $H\left(\frac{\pi\mu}{2}\right)=\left(H\left(\frac{\pi}{2}\right)\right)^{\mu}=I$. Now consider $\sum\limits_{s=1}^{q}\mu_{s}F_{s}$ to be the spectral decomposition of adjacency matrix of $G$. As the graph $G$ is assumed to be integral, the eigenvalues $\mu_s$ of $G$ are integers. By Proposition \ref{aa}, the transition matrix of the graph $X(\C)\times G$ at $\frac{\pi}{2}$ is obtained as
$$\sum\limits_{s=1}^{q} H\left(\frac{\pi\mu_s}{2}\right)\otimes F_{s}=I\otimes\sum\limits_{s=1}^{q} F_{s}=I\otimes I=I,$$ where the identity matrices have appropriate orders. This proves our claim.
\end{proof}

\section{Periodicity and PST on gcd-Graphs}
Now we investigate gcd-graphs for periodicity and PST. First we construct some periodic gcd-graphs which are not necessarily connected. Also we find some gcd-graphs exhibiting PST and then we club them to obtain connected gcd-graph having PST. First consider the following characterization of PST in vertex-transitive graphs as given in \cite{god1}.
\begin{theorem}\cite{god1}\label{c1}
Suppose $G$ is a connected vertex-transitive graph with vertices $u$ and $v$, and perfect state transfer from $u$ to $v$ occurs at time $\tau$. Then the transition matrix of $G$ is a scalar multiple of a permutation matrix of order two and no fixed points, and it lies in the center of the automorphism group of $G$.
\end{theorem}
Consequently, if a vertex-transitive graph admits PST then it must have an even number of vertices. We therefore have the following obvious characterization for PST in gcd-graphs.
\begin{cor}\label{c1c}
A gcd-graph over a group of odd order does not exhibit perfect state transfer.
\end{cor}
The eigenvalues and the corresponding eigenvectors of a Cayley graph over an abelian group are well known. In \cite{wal1}, it is shown that the eigenvectors are independent of the connection set. Consider two symmetric subsets $S_1,S_2$ in $\Gamma$. So the set of eigenvectors of both graphs $Cay(\Gamma, S_1 )$ and $Cay(\Gamma, S_2 )$ can be chosen to be equal. Hence we have the following result.
\begin{prop}\label{3ab}
If $S_1$ and $S_2$ are symmetric subsets of an abelian group $\Gamma$ then adjacency matrices of the Cayley garphs $Cay(\Gamma, S_1 )$ and $Cay(\Gamma, S_2 )$ commute.
\end{prop}
The following result allows us to find the transition matrix of a Cayley graph, which is union of two edge disjoint Cayley graphs over an abelian group.
\begin{prop}\label{3a}
Let $\Gamma$ be a finite abelian group and consider two disjoint and symmetric subsets $S,T\subset\Gamma$. Suppose the transition matrices of $Cay(\Gamma, S)$ and $ Cay(\Gamma, T)$ are $H_{S}(t)$ and $H_{T}(t)$, respectively. Then $Cay(\Gamma, S\cup T)$ has the transition matrix $H_{S}(t)H_{T}(t).$
\end{prop}

\begin{proof}
The group $\Gamma$ is a finite abelian group. By Proposition \ref{3ab}, the adjacency matrices of the graphs $Cay(\Gamma, S)$ and $ Cay(\Gamma, T)$ commute. We know that for any two square matrices $A, B$ with $AB=BA$, we have $\exp{(A+B)}=\exp{(A)}\exp{(B)}$. Using this we get the desired result.
\end{proof}
Suppose the prime factorization of an integer $n(\geq2)$ is $n = p_1\ldots p_k,$ where the primes are not necessarily distinct. In \cite{wal}, the authors showed that every gcd-graph with $n$ vertices is isomorphic to a gcd-graph over $\Gamma= \Zl_{p_1}\oplus\ldots\oplus\Zl_{p_k}$. If $n$ is a power of $2$ then the gcd-graph is actually a cubelike graph. The next theorem is thus a special case to that result. Still we include the result so as to have a definite structure of the connection set of the cubelike graph, which will be used to characterize PST on gcd-graphs. We make use of some techniques from \cite{wal} and consider looped graphs.\par
Assume that $X(\C_1)$ and $X(\C_2)$ are two cubelike graphs. It is easy to see that there is a natural isomorphism between $X(\C_1)\times X(\C_2)$ and $X(\C_1\times \C_2)$. 

\begin{theorem}\cite{wal}\label{6a}
A gcd-graph over an abelian group of order $2^n$ is isomorphic to a cubelike graph.
\end{theorem}

\begin{proof}
Consider an abelian group $\Gamma=\Zl_{2^{n_1}}\oplus\ldots\oplus\Zl_{2^{n_r}}$. For each set of divisor tuples $\D$, we show that $Cay(\Gamma,S_{\Gamma}(\D))$ is isomorphic to a cubelike graph. Let $\d=\left(d_1,\ldots,d_k,\ldots,d_r\right)\in \D$. For $i=1,\ldots,r$ we have $d_i=2^{k_i}$ for some $k_i\leq n_i$. If $\x=\left(x_1,\ldots,x_r\right),\y=\left(y_1,\ldots,y_r\right)\in\Gamma$ and $\n=\left(2^{n_1},\ldots,2^{n_r}\right)$ then $gcd(\x-\y,\n)=\d$ if and only if $gcd(x_i-y_i,2^{n_i})=d_i$. That is $\x\sim \y$ in $Cay(\Gamma,S_{\Gamma}(\d))$ if and only if $x_i\sim y_i$ in $Cay(\Zl_{2^{n_i}},S_{\Zl_{2^{n_i}}}(d_i))$. Note that, if $d_i=2^{n_i}$ then $Cay(\Zl_{2^{n_i}},S_{\Zl_{2^{n_i}}}(d_i))$ is the graph with loops at each of its vertices and no other edges. We therefore have the following:
$$Cay(\Gamma,S_{\Gamma}(\d))\cong Cay(\Zl_{2^{n_1}},S_{\Zl_{2^{n_1}}}(d_1))\times\ldots\times Cay(\Zl_{2^{n_r}},S_{\Zl_{2^{n_r}}}(d_r)).$$
Now for a fixed $i$, if $z\in\Zl_{2^{n_i}}$ then there exists a unique 2-adic representation
\begin{eqnarray*}
z=\sum\limits_{j=0}^{n_i-1}z_{j}2^{j}, \text{ where } z_{j}\in\left\lbrace 0,1 \right\rbrace \text{ for } j=0,1,\ldots, n_i-1.
\end{eqnarray*}
We write $\tilde{\textbf{z}}=\left( z_0,\ldots,z_{n_i-1}\right)\in \Zl_2^{n_i}$. We show that the map $z\mapsto \tilde{\textbf{z}}$ gives an isomorphism of $Cay(\Zl_{2^{n_i}},S_{\Zl_{2^{n_i}}}(d_i))$ to the cubelike graph $X(\textbf{C}_{d_i})$ over $\Zl_2^{n_i}$ where
$$\textbf{C}_{d_i}=\left\lbrace \left(c_0, c_1,\ldots,c_{n_i-1}\right)\in\Zl_2^{n_i}: c_j=0 \text{ for every } j< k_i \text{ and } c_{k_i}=1\right\rbrace,\;\text{where }d_i=2^{k_i}.$$
Note that for $j>k_i$, the value of $c_j$ can be either $0$ or $1$.
It is enough to show that $u\sim v$ in $Cay(\Zl_{2^{n_i}},S_{\Zl_{2^{n_i}}}(d_i))$ if and only if $\tilde{\ul}-\tilde{\vl}\in \textbf{C}_{d_i}$. Now $u\sim v$ in $Cay(\Zl_{2^{n_i}},S_{\Zl_{2^{n_i}}}(d_i))$ if and only if $gcd(u-v,2^{n_i})=2^{k_i}$. Observe that $gcd(u-v,2^{n_i})=2^{k_i}$ if and only if $u_{j}-v_{j}=0$ for every $j< k_i$ and $u_{k_i}-v_{k_i}=1$. Thus we have
$$Cay(\Zl_{2^{n_i}},S_{\Zl_{2^{n_i}}}(d_i))\cong X(\textbf{C}_{d_i}),\text{ for each }i=1,\ldots,r.$$
Therefore for $\d\in \D$ we find that
$$Cay(\Gamma,S_{\Gamma}(\d))\cong X(\textbf{C}_{d_1})\times\ldots\times X(\textbf{C}_{d_r})\cong X(\textbf{C}_{d_1}\times\ldots\times \textbf{C}_{d_r}).$$
The isomorphism that we exhibited between the vertices of $Cay(\Gamma,S_{\Gamma}(\d))$ and $X(\textbf{C}_{d_1}\times\ldots\times \textbf{C}_{d_r})$ will work for all divisors $\d\in\D$. Hence we have
$$Cay(\Gamma,S_{\Gamma}(\D))\cong X(\textbf{C}),\text{ where }\textbf{C}=\bigcup\limits_{\d\in \D}\textbf{C}_{d_1}\times\ldots\times \textbf{C}_{d_r}.$$ This completes the proof.
\end{proof}

Assume that $\Gamma$ is a finite abelian group. We write $\Gamma=\Gamma_1\oplus \Gamma_2$, where $\Gamma_1$ is an abelian group of order $2^n$ and $\Gamma_2$ is an abelian group of odd order.  Also, consider the cyclic group decomposition of $\Gamma_1$ and $\Gamma_2$ as follows:
$$\displaystyle{\Gamma_1=\Zl_{2^{n_1}}\oplus\ldots\oplus\Zl_{2^{n_r}},\; \Gamma_2=\Zl_{p_{1}^{m_{1}}}\oplus\ldots\oplus\Zl_{p_s^{m_s}}.}$$
We write $\m_{\Gamma}:=\left(2^{n_1},\ldots,2^{n_r},p_1^{m_1},\ldots,p_s^{m_s}\right)$ associated to the group $\Gamma$. In what follows, we will always consider a finite abelian group $\Gamma$ in the form $\Gamma_1\oplus\Gamma_2$ as described above. We have the following lemma on the structure of certain gcd-graphs.

\begin{lem}\label{6b}
Let $\Gamma_1=\Zl_{2^{n_1}}\oplus\ldots\oplus\Zl_{2^{n_r}}$ and $\Gamma_2=\Zl_{p_{1}^{m_{1}}}\oplus\ldots\oplus\Zl_{p_s^{m_s}}$, $p_i>2$, be such that $\Gamma=\Gamma_1\oplus\Gamma_2$. Assume that $d_{r+1},\ldots,d_{r+s}$ are fixed divisors of $p_{1}^{m_{1}},\ldots,p_{s}^{m_{s}}$, respectively. Consider a set of divisor tuples $\D$ such that $\m_{\Gamma}\notin \D$. If the last $s$ components of each $\d\in \D$ are $d_{r+1},\ldots,d_{r+s}$ then there exist a cubelike graph $X(\C)$ such that $$Cay(\Gamma,S_{\Gamma}(\D))\cong X(\C)\times Cay(\Gamma_2,S_{\Gamma_2}\left(\left\lbrace\left(d_{r+1},\ldots,d_{r+s}\right)\right\rbrace\right)).$$
\end{lem}

\begin{proof}
Assume that $\D^*=\left\lbrace\left(d_1,\ldots,d_r\right):\d=\left(d_1,\ldots,d_r,\ldots,d_{r+s}\right)\in \D\right\rbrace$. Therefore
$$Cay(\Gamma,S_{\Gamma}(\D))\cong Cay(\Gamma_1,S_{\Gamma_1}(\D^*))\times Cay(\Gamma_2,S_{\Gamma_2}\left(\left\lbrace\left(d_{r+1},\ldots,d_{r+s}\right)\right\rbrace\right)).$$
By Theoren \ref{6a}, there exist a cubelike graph $X(\C)$ such that $Cay(\Gamma_1,S_{\Gamma_1}(\D^*))\cong X(\C)$. Hence we have the desired result.
\end{proof}
Now we find a sufficient condition for a gcd-graph to exhibit periodicity at $\frac{\pi}{2}$. Let $\mathscr{D}'$  be the collection of all divisors sets $\D$ satisfying the conditions of Lemma \ref{6b} as well as the following two conditions:
\begin{enumerate}
\item the sum of the elements in $\C$ is zero; and
\item $|\C|\equiv 0\;(\text{mod}\; 4)$ whenever $d_{r+i}<p_i^{m_i}$ for some $i=1,\ldots,s$.
\end{enumerate}
Now suppose $\mathscr{D}$ is the collection of all disjoint union of members of $\mathscr{D}'$. The following result determines a class of periodic gcd-graphs.

\begin{theorem}\label{6c}
If $\Gamma$ is a finite abelian group and $\D\in\mathscr{D}$ then $Cay(\Gamma,S_{\Gamma}(\D))$ is periodic at $\frac{\pi}{2}$.
\end{theorem}
\begin{proof}
Let $\D=\D_1\cup\ldots\cup\D_k\in\mathscr{D}$, where $\D_l\in\mathscr{D}'$ for all $l=1,\ldots,k$. Suppose $d_{r+j}$ is a fixed divisor of $p_{j}^{m_j}$ for $j=1,\ldots,s$. For a fixed $l$ assume that the last $s$ components of each $\d\in \D_l$ are $d_{r+1},\ldots,d_{r+s}$. By using Lemma \ref{6b} we find that $Cay(\Gamma,S_{\Gamma}(\D_l))$ is isomorphic to $X(\C)\times Cay(\Gamma_2,S_{\Gamma_2}\left(\left\lbrace\left(d_{r+1},\ldots,d_{r+s}\right)\right\rbrace\right)).$
Consider the following two cases.\\
\textbf{Case I:} ($d_{r+j}<p_{j}^{m_j}$ for some $j$) By our assumption on $\mathscr{D}$, the connection set $\C$ in $X(\C)$ in this case is such that $|\C|\equiv 0\;(\text{mod}\; 4)$ and the sum of the elements in $\C$ is zero. Therefore by Proposition \ref{2d}, the transition matrix of $Cay(\Gamma,S_{\Gamma}(\D_l))$ is the identity matrix at $\frac{\pi}{2}$.\\
\textbf{Case II:} ($d_{r+j}=p_{j}^{m_j}$ for all $j$) In this case the sum of the elements of $\C$ in $X(\C)$ is zero. Note that $Cay(\Gamma_2,S_{\Gamma_2}\left(\left\lbrace\left(d_{r+1},\ldots,d_{r+s}\right)\right\rbrace\right))$ has loops at each of its vertices and no more edges. Recall that each loop contributes $1$ to the adjacency matrix according to our convention. If $A$ is the adjacency matrix of $X(\C)$ then by using Lemma \ref{6b} we find that the adjacency matrix of $Cay(\Gamma,S_{\Gamma}(\D_l))$ is $A\otimes I$. The transition matrix of $Cay(\Gamma,S_{\Gamma}(\D_l))$ can be calculated as $$\displaystyle{\exp{\left(it(A\otimes I)\right)}=\exp{(itA)}\otimes I}.$$
Observe that $\displaystyle{\exp{(itA)}}$ is the transition matrix of $X(\C)$. By Theorem \ref{2c}, the graph $X(\C)$ is periodic at $\frac{\pi}{2}$ and therefore $Cay(\Gamma,S_{\Gamma}(\D_l))$ is also periodic at $\frac{\pi}{2}$.\\
In both the cases the graph $Cay(\Gamma,S_{\Gamma}(\D_l))$ is periodic at $\frac{\pi}{2}$. Finally, for $\D\in\mathscr{D}$, we apply Proposition \ref{3a} to have the desired result.
\end{proof}
We illustrate this by the following example. Here we find a periodic graph over $\Zl_4\oplus\Zl_2\oplus\Zl_3$.
\begin{ex}\label{eg0}
Consider $\Gamma=\Zl_4\oplus\Zl_2\oplus\Zl_3$ and the set of divisor tuples $\D=\left\lbrace(1,1,1),(1,2,1)\right\rbrace$ of $(4,2,3)$. We show that the graph $Cay(\Gamma,S_{\Gamma}(\D))$ is periodic at $\frac{\pi}{2}$. Suppose $\Gamma'=\Zl_4\oplus\Zl_2$ and consider $\D^*=\left\lbrace(1,1),(1,2)\right\rbrace$. Now we follow Theorem \ref{6a} to find the cubelike graph isomorphic to $Cay(\Gamma',S_{\Gamma'}(\D^*))$. For $\d=(d_1,d_2)=(1,1)$, we set $\C_{d_1}=\left\lbrace(1,0),(1,1)\right\rbrace$ and $\C_{d_2}=\left\lbrace(1)\right\rbrace$. Therefore $\C_{d_1}\times \C_{d_2}=\left\lbrace(1,0,1),(1,1,1)\right\rbrace$. Similarly, for $\d'=(d'_1,d'_2)=(1,2)$, we find that $\C_{d'_1}\times \C_{d'_2}=\left\lbrace(1,0,0),(1,1,0)\right\rbrace$. Hence $Cay(\Gamma',S_{\Gamma'}(\D^*))$ is isomorphic to $X(\C)$ which has the connection set $\C=\left(\C_{d_1}\times \C_{d_2}\right)\cup \left(\C_{d'_1}\times \C_{d'_2}\right)=\left\lbrace(1,0,1),(1,1,1),(1,0,0),(1,1,0)\right\rbrace$. Clearly $|\C|\equiv 0\;(\text{mod}\;4)$ and the sum of the elements in $\C$ is $\o$ in $\Zl_2^3$. Therefore $\D\in\mathscr{D}$ and hence by Theorem \ref{6c}, the graph $Cay(\Gamma,S_{\Gamma}(\D))$ is periodic at $\frac{\pi}{2}$.
\end{ex}
In this way we can construct many gcd-graphs which are periodic. Our next motive is to add extra edges to these graphs so that the graphs exhibit PST. Generally periodicity and PST are considered for connected graphs. Observe that whenever $\D$ generates the whole group $\Gamma$ then the associated gcd-graph is connected. In the following result we find a sufficient condition for a gcd-graph to exhibit PST.\par
 Consider a set of divisor tuples $\D$ of $\m_{\Gamma}$ with $\m_{\Gamma}\notin \D$. Also, assume that for each tuple $d=\left(d_1,\ldots,d_r,\ldots,d_{r+s}\right)$ in $\D$, the divisors $d_{r+j}=p_{j}^{m_j}$ for $j=1,\ldots,s$. Let $X(\C)$ be the cubelike graph as in Lemma \ref{6b} associated to $Cay(\Gamma,S_{\Gamma}(\D))$. We denote the set of all $\D$ such that sum of the elements in $\C$ is non-zero by $\mathscr{\tilde{D}}$.

\begin{theorem}\label{6d}
Let $\Gamma$ be a finite abelian group. Suppose $\D_1\in\mathscr{D}$ and $\D_2\in\tilde{\mathscr{D}}$ and $\D_1\cap\D_2=\emptyset$. If $\D= \D_1\cup \D_2$ so that $\D$ generates $\Gamma$ then $Cay(\Gamma,S_{\Gamma}(\D))$ is connected and admits perfect state transfer at $\frac{\pi}{2}$.
\end{theorem}
\begin{proof}
If $\D$ generates the group $\Gamma$ then the graph $Cay(\Gamma,S_{\Gamma}(\D))$ is clearly connected. By Theorem \ref{6c}, the graph $Cay(\Gamma,S_{\Gamma}(\D_1))$ is periodic at $\frac{\pi}{2}$. Thus by Proposition \ref{3a}, it is enough to show that $Cay(\Gamma,S_{\Gamma}(\D_2))$ exhibits PST at $\frac{\pi}{2}$.\par
Using Lemma \ref{6b} we find that
$$Cay(\Gamma,S_{\Gamma}(\D_2))\cong X(\C)\times Cay(\Gamma_2,S_{\Gamma_2}\left(\left\lbrace\left(p_{1}^{m_1},\ldots,p_{s}^{m_s}\right)\right\rbrace\right)).$$
Since $G=Cay(\Gamma_2,S_{\Gamma_2}\left(\left\lbrace\left(p_{1}^{m_1},\ldots,p_{s}^{m_s}\right)\right\rbrace\right))$ has loops at each of its vertices and no more edges, the adjacency matrix of $G$ is $I$. Suppose $X(\C)$ has the adjacency matrix $A$. The adjacency matrix of $Cay(\Gamma,S_{\Gamma}(\D_2))$ is therefore $A\times I$. Now the transition matrix of $Cay(\Gamma,S_{\Gamma}(\D_2))$ can be calculated as $\displaystyle{\exp{(itA\times I)}=\exp{(itA)}\times I}$. Observe that $\displaystyle{\exp{(itA)}}$ is the transition matrix of $X(\C)$. Since $\D_2\in\tilde{\mathscr{D}}$, the sum of the elements in $\C$ is non-zero. Therefore, by Theorem \ref{2c}, the cubelike graph $X(\C)$ exhibits PST at $\frac{\pi}{2}$. Hence $Cay(\Gamma,S_{\Gamma}(\D_2))$ also exhibits PST at $\frac{\pi}{2}$.
\end{proof}
We illustrate Theorem \ref{6d} by the following example.
\begin{ex}\label{eg1}
Suppose $\Gamma=\Zl_4\oplus\Zl_2\oplus\Zl_3$ and let $\D=\left\lbrace(1,1,1),(1,2,1),(2,2,3),(4,1,3)\right\rbrace$. We show that the graph $Cay(\Gamma,S_{\Gamma}(\D))$ admits PST at $\frac{\pi}{2}$. Let $\D_1=\left\lbrace(1,1,1),(1,2,1)\right\rbrace$. We already have in Example \ref{eg0} that $\D_1\in\mathscr{D}$. Now consider $\D_2=\left\lbrace(2,2,3),(4,1,3)\right\rbrace$. Here the connection set $\C$ in $X(\C)$ associated to $Cay(\Gamma,S_{\Gamma}(\D_2))$ can be evaluated as $\C=\left\lbrace (0,1,0),(0,0,1)\right\rbrace.$ Thus we have $\D_2\in\tilde{\mathscr{D}}$. Note here that $\D$ generates $\Gamma$. Hence, by using Theorem \ref{6d}, we find that $Cay(\Gamma,S_{\Gamma}(\D))$ is connected and exhibits PST at $\frac{\pi}{2}$.
\end{ex}

Now we find a characterization of gcd-graphs having perfect state transfer. We show that if $|\Gamma|\equiv 0\;(\text{mod}\; 4)$ then there is a gcd-graph over $\Gamma$ having perfect state transfer. Consider the following two results.

\begin{lem}\label{n1}
If $\Zl_8$ is a subgroup of $\Gamma$ then there exists a connected gcd-graph over $\Gamma$ exhibiting perfect state transfer at $\frac{\pi}{2}$.
\end{lem}
\begin{proof}
Suppose $\Gamma_1=\Zl_{2^{n_1}}\oplus\ldots\oplus\Zl_{2^{n_r}}$, $n_1>2$ and $\Gamma_2=\Zl_{p_{1}^{m_{1}}}\oplus\ldots\oplus\Zl_{p_s^{m_s}}$, $p_i>2$ for $1\leq i\leq s$ be such that $\Gamma=\Gamma_1\oplus\Gamma_2$. Consider the set of divisors $$\D=\left\lbrace \left(d_1,\ldots, d_{r+s}\right): d_1=1 \text{ and for } i\geq 2,\; d_i=1 \text{ for atmost one } i, \text{ otherwise } d_i=0\right\rbrace.$$
Here the set $\D$ generates the group $\Gamma$ and hence $Cay(\Gamma,S_{\Gamma}(\D))$ is connected. Now for $\d\in \D$, we show that the cubelike graph $X(\C)$ associated to $Cay(\Gamma,S_{\Gamma}(\left\lbrace \d \right\rbrace))$ (as in Lemma \ref{6b}) has the connection set $\C$ which has the property that $|\C|\equiv 0\;(\text{mod}\;4)$ and the sum of the elements in $\C$ is $\o$.\par
Here notice that $d_1=1=2^0$ and $n_1>2$. Therefore $$\C_{d_1}=\left\lbrace\left(c_0,\ldots,c_{n_1-1}\right)\in\Zl_2^{n_1}: c_0=1 \text{ and for } i\geq 1,\; c_i=0 \text{ or }1 \right\rbrace.$$
Observe that $\C_{d_1}$ has $2^{n_1-1}$ elements and hence $|\C_{d_1}|\equiv 0\;(\text{mod}\;4)$. Also it is clear that the sum of the elements in $\C_{d_1}$ is $\o$. Notice that for any subset $S$ of $\Zl_2^k$, we have $|\C_{d_1}\times S|\equiv 0\;(\text{mod}\;4)$ and the sum of the elements in $\C_{d_1}\times S$ is $\o$. Here the connection set $\C$ in $X(\C)$ is given by $\C=\C_{d_1}\times\C_{d_2}\times\ldots\times\C_{d_r}$. So $|\C|\equiv 0\;(\text{mod}\;4)$ and the sum of the elements in $\C$ is $\o$. Hence we find that $\left\lbrace d\right\rbrace\in\mathscr{D}$ for each $d\in\D$ and therefore $\D\in\mathscr{D}$.\par
Now consider $\D'=\left\lbrace\left(2^{n_1-1},0,\ldots, 0\right)\right\rbrace$. Here we show that $\D'\in\tilde{\mathscr{D}}$. Let $X(\C')$ be the cubelike graph associated to $Cay(\Gamma,S_{\Gamma}(\D'))$ (as in Lemma \ref{6b}). Note that $\C'_{d_1}=\left\lbrace\left(0,\ldots,0,1\right)\right\rbrace$ and for $2\leq i\leq r$ we have $\C'_{d_i}=\left\lbrace\left(0,\ldots,0\right)\right\rbrace$. So the set $\C'$, which is $\C'_{d_1}\times\ldots\times\C'_{d_r}$, contains exactly one element (non zero) and hence $\D'\in\tilde{\mathscr{D}}$. Finally by Theorem \ref{6d} we find that $Cay(\Gamma,S_{\Gamma}(\D\cup\D'))$ is connected and exhibits perfect state transfer at $\frac{\pi}{2}$.
\end{proof}

\begin{lem}\label{n2}
Let $\Zl_4$ be a subgroup of $\Gamma$ and suppose $\Zl_8$ is not a subgroup of $\Gamma$. Then there exists a connected gcd-graph over $\Gamma$ exhibiting perfect state transfer at $\frac{\pi}{2}$.
\end{lem}
\begin{proof}
Suppose $\Gamma_1=\Zl_{2^{n_1}}\oplus\ldots\oplus\Zl_{2^{n_r}}$, $n_1=2$ and $\Gamma_2=\Zl_{p_{1}^{m_{1}}}\oplus\ldots\oplus\Zl_{p_s^{m_s}}$, $p_i>2$ for $1\leq i\leq s$ be such that $\Gamma=\Gamma_1\oplus\Gamma_2$. Consider the set of divisors $$\D=\left\lbrace \left(d_1,\ldots, d_{r+s}\right): d_1\in\left\lbrace 0,1,2\right\rbrace \text{ and for } i\geq 2,\; d_i=1 \text{ for only one } i, \text{ otherwise } d_i=0\right\rbrace.$$
Here the set $\D$ generates the group $\Gamma$ and hence $Cay(\Gamma,S_{\Gamma}(\D))$ is connected. Here notice that if $d_1=0$ then $\C_{d_1}=\left\lbrace\left(0,0\right)\right\rbrace$ and if $d_1=1$ then $\C_{d_1}=\left\lbrace\left(1,0\right),(1,1)\right\rbrace$ and when $d_1=2$ we have $\C_{d_1}=\left\lbrace\left(0,1\right)\right\rbrace.$ Now for a fixed divisor tuple $(d_2,\ldots,d_{r+s})$ consider the set of divisor tuples $\D_1=\left\lbrace 0,1,2\right\rbrace\times\left\lbrace(d_2,\ldots,d_{r+s})\right\rbrace.$ Suppose $X(\C)$ is the cubelike graph associated to $Cay(\Gamma,S_{\Gamma}(\D_1))$ (as in Lemma \ref{6b}). The connection set $\C$ is therefore
$$\C=\left(\bigcup\limits_{d_1\in\left\lbrace 0,1,2\right\rbrace}C_{d_1}\right)\times C_{d_2}\times\ldots\times C_{d_r}.$$ 
Observe that $|\C|\equiv 0\;(\text{mod}\;4)$ and the sum of the elements in $\C$ is $\o$. Hence we find that $\D\in\mathscr{D}$.\par
Now consider $\D'=\left\lbrace\left(2,0,\ldots, 0\right)\right\rbrace$. Here we show that $\D'\in\tilde{\mathscr{D}}$. Let $X(\C')$ be the cubelike graph associated to $Cay(\Gamma,S_{\Gamma}(\D'))$ (as in Lemma \ref{6b}). Note that $\C'_{d_1}=\left\lbrace\left(0,1\right)\right\rbrace$ and for $2\leq i\leq r$ we have $\C'_{d_i}=\left\lbrace\left(0,\ldots,0\right)\right\rbrace$. So the set $\C'$, which is $\C'_{d_1}\times\ldots\times\C'_{d_r}$, contains exactly one element and it is non zero and hence $\D'\in\tilde{\mathscr{D}}$. Finally by Theorem \ref{6d} we find that $Cay(\Gamma,S_{\Gamma}(\D\cup\D'))$ is connected and exhibits perfect state transfer at $\frac{\pi}{2}$.
\end{proof}

\begin{lem}\label{n3}
Let $\Gamma$ be a group such that $|\Gamma|\equiv 0\;(\text{mod}\;4)$. If $\Zl_2$ is a subgroup of $\Gamma$ and $\Zl_4$ is not a subgroup of $\Gamma$ then there exists a connected gcd-graph over $\Gamma$ exhibiting perfect state transfer at $\frac{\pi}{2}$.
\end{lem}
\begin{proof}
Suppose $\Gamma_1=\Zl_{2^{n_1}}\oplus\ldots\oplus\Zl_{2^{n_r}}$, $n_i=1$ for $1\leq i\leq r$ and $\Gamma_2=\Zl_{p_{1}^{m_{1}}}\oplus\ldots\oplus\Zl_{p_s^{m_s}}$, $p_i>2$ for $1\leq i\leq s$ be such that $\Gamma=\Gamma_1\oplus\Gamma_2$. Consider the set of divisors $$\D=\left\lbrace \left(d_1,\ldots, d_{r+s}\right): d_1,d_2\in\left\lbrace 0,1\right\rbrace \text{ and for } i\geq 3,\; d_i=1 \text{ for only one } i, \text{ otherwise } d_i=0\right\rbrace.$$
For $1\leq i\leq r$, if $d_i=0$ then $\C_{d_i}=\left\lbrace\left(0\right)\right\rbrace$ and when $d_i=1$ we have $\C_{d_i}=\left\lbrace\left(1\right)\right\rbrace.$ Now for a fixed divisor tuple $(d_3,\ldots,d_{r+s})$ consider $\D_1=\left\lbrace 0,1\right\rbrace\times\left\lbrace 0,1\right\rbrace\times\left\lbrace(d_3,\ldots,d_{r+s})\right\rbrace.$ Suppose $X(\C)$ is the cubelike graph associated to $Cay(\Gamma,S_{\Gamma}(\D_1))$ (as in Lemma \ref{6b}). The connection set $\C$ is therefore
$$\C=\left(\bigcup\limits_{d_1\in\left\lbrace 0,1\right\rbrace}C_{d_1}\right)\times\left(\bigcup\limits_{d_2\in\left\lbrace 0,1\right\rbrace}C_{d_2}\right)\times C_{d_3}\times\ldots\times C_{d_r}.$$ 
Observe that $|\C|\equiv 0\;(\text{mod}\;4)$ and the sum of the elements in $\C$ is $\o$. Hence we find that $\D\in\mathscr{D}$.\par
Now consider $\D'=\left\lbrace\left(1,0,0,\ldots, 0\right),\left(0,1,0,\ldots, 0\right)\right\rbrace$. Here it is clear that $\D'\in\tilde{\mathscr{D}}$. Also note that the set $\D\cup\D'$ generates the group $\Gamma$ and hence $Cay(\Gamma,S_{\Gamma}(\D))$ is connected. Finally, by Theorem \ref{6d} we find that $Cay(\Gamma,S_{\Gamma}(\D\cup\D'))$ is connected and exhibits PST at $\frac{\pi}{2}$.
\end{proof}
We now combine Lemma \ref{n1}, Lemma \ref{n2} and Lemma \ref{n3} to state the following theorem.
\begin{theorem}
Let an abelian group $\Gamma$ be such that $|\Gamma|\equiv 0\;(\text{mod}\;4)$. Then there exists a connected gcd-graph over $\Gamma$ exhibiting perfect state transfer at $\frac{\pi}{2}$.
\end{theorem}

We have already seen that gcd-graphs over an abelian groups of odd order do not exhibit PST. In the following theorem we find a necessary and sufficient condition for a class of gcd-graphs to be periodic at $\pi$. Here we add some more edges to a graph $Cay(\Gamma,S_{\Gamma}(\D))$ admitting perfect PST at $\frac{\pi}{2^{k}},k\in\Nl$ and observe the behavior of the transition matrix. Using this, we will find some gcd-graphs not allowing PST at $\frac{\pi}{2^{k}},$ for all $k\in\Nl$. 

\begin{theorem}\label{6e}
Let $\Gamma$ be as defined in Lemma \ref{6b} and assume that $Cay(\Gamma,S_{\Gamma}(\D))$ admits perfect state transfer at $\frac{\pi}{2^{k}},$ for some $k\in\Nl$. Suppose $d_{r+1},\ldots,d_{r+s}$ are fixed divisors of $p_{1}^{m_{1}},\ldots,p_{s}^{m_{s}}$, respectively. Consider $\D'=\left\lbrace \d=\left(d_1,\ldots,d_{r+1},\ldots,d_{r+s}\right)\in\Gamma\;:\;\d \text{ divides } \m_{\Gamma}\right\rbrace,$ so that $\D$ and $\D'$ are disjoint. Also suppose $Cay(\Gamma,S_{\Gamma}(\D'))\cong X(\C')\times G$ with $|\C'|\not\equiv 0\;(\text{mod}\;2)$. The graph $Cay(\Gamma,S_{\Gamma}(\D\cup \D'))$ is periodic at $\pi$ if and only if the eigenvalues of the integral graph $G$ have same parity.
\end{theorem}
\begin{proof}
By using Theorem \ref{c1} we find that if a vertex transitive graph exhibits PST at $\tau$ then for all $k\in\Nl$, the transition matrix $H\left(2^{k}\tau\right)$ is a scalar multiple of identity. Since $Cay(\Gamma,S_{\Gamma}(\D))$ admits PST at $\frac{\pi}{2^{k}},$ the  associated transition matrix at $\pi$ is a scalar multiple of the identity matrix. By applying Proposition \ref{3a}, the transition matrix of $Cay(\Gamma,S_{\Gamma}(\D\cup \D'))$ can be evaluated as the product of transition matrices of $Cay(\Gamma,S_{\Gamma}(\D))$ and $Cay(\Gamma,S_{\Gamma}(\D'))$. Hence $Cay(\Gamma,S_{\Gamma}(\D\cup \D'))$ is periodic at $\pi$ if and only if $Cay(\Gamma,S_{\Gamma}(\D'))$ is periodic at $\pi$.\par
Also we have $Cay(\Gamma,S_{\Gamma}(\D'))\cong X(\C')\times G$. If the sum of the elements in $\C'$ is $\sigma$ then by Lemma \ref{2b}, the transition matrix of $X(\C')$ at $\frac{\pi}{2}$ can be calculated as $i^{|\C'|}P_{\sigma}$. At $\tau=\pi$, the transition matrix becomes $[i^{|\C'|}P_{\sigma}]^2=i^{2|\C'|}I$ as $P_{\sigma}^2=I$. If $\sum\limits_s\mu_s F_s$ is the spectral decomposition of adjacency matrix of $G$ then by Proposition \ref{aa} the transition matrix of $X(\C')\times G$ becomes $$\sum\limits_s \left(i^{2|\C'|}\right)^{\mu_s}I\otimes F_s=I\otimes\sum\limits_s i^{2|\C'|\mu_s} F_s.$$
If $Cay(\Gamma,S_{\Gamma}(\D'))$ is periodic at $\pi$ then $\exists \gamma\in\Cl$ with $|\gamma|=1$ such that $$I\otimes\sum\limits_s i^{2|\C'|\mu_s} F_s=\gamma I,$$
where the identity matrices have appropriate orders. Note that $F_s^2=F_s$ and $F_r F_s=0$ for $r\neq s$. Multiplying both sides of the above  equation by $I\otimes F_s$ we find that $i^{2|\C'|\mu_s}=\gamma$ for all $s$. If $\mu_s$ and $\mu_{s'}$ are two distinct eigenvalues of $G$ then we have $e^{i|\C'|\left(\mu_s-\mu_{s'}\right)\pi}=1$. By our assumption $|\C'|\not\equiv 0\;(\text{\text{mod}}\;2)$ and therefore the eigenvalues $\mu_s$ and $\mu_{s'}$ must have same parity.\par
Conversely, suppose the eigenvalues of $G$ have same parity and let $i^{2|\C'|\mu_s}=\gamma$. Then for any other eigenvalue $\mu_{s'}$, we see that $i^{2|\C'|\mu_{s'}}=i^{2|\C'|\mu_{s}}\cdot i^{2|\C'|\left(\mu_{s'}-\mu_s\right)}=\gamma$. Thus $i^{2|\C'|\mu_s}=\gamma$ for all $\mu_s$. Therefore $I\otimes\sum\limits_s i^{2|\C'|\mu_s} F_s=\gamma I$ and hence $Cay(\Gamma,S_{\Gamma}(\D'))$ is periodic at $\pi$. This in turn implies that $Cay(\Gamma,S_{\Gamma}(\D\cup \D'))$ is periodic at $\pi$. 
\end{proof}
We thus have a necessary condition for PST at $\frac{\pi}{2^{k}},\;k\in\Nl$ in a class of gcd-graphs given in Theorem \ref{6e}. We include this as a corollary.
\begin{cor}\label{6f}
Suppose the condition of Theorem \ref{6e} holds. If $Cay(\Gamma,S_{\Gamma}(\D\cup \D'))$ admits perfect state transfer at $\frac{\pi}{2^{k}},\;k\in\Nl$ then the eigenvalues of $G$ have same parity.
\end{cor}
In the following example we use Corollary \ref{6f} to find a gcd-graph not having PST at $\frac{\pi}{2^{k}},$ for all $k\in\Nl$.
\begin{ex}
Suppose $\Gamma=\Zl_4\oplus\Zl_2\oplus\Zl_3$ and consider $\D=\left\lbrace(1,1,1),(1,2,1),(2,2,3),(4,1,3)\right\rbrace$ and $\D'=\left\lbrace(2,2,1)\right\rbrace$. In Example \ref{eg1} we already found  that $Cay(\Gamma,S_{\Gamma}(\D))$ admits PST at $\frac{\pi}{2}$. By using Lemma \ref{6b}, we find that $Cay(\Gamma,S_{\Gamma}(\D'))\cong X(\C)\times G$, where $G=Cay(\Zl_3,S_{\Zl_3}(1))$. Note that $Cay(\Zl_3,S_{\Zl_3}(1))$ is actually the complete graph $K_3$ whose eigenvalues are $-1,-1$ and $2$. Clearly the eigenvalues of $G$ does not have same parity. Hence, by Corollary \ref{6f}, we conclude that $Cay(\Gamma,S_{\Gamma}(\D\cup \D'))$ does not exhibit PST at $\frac{\pi}{2^{k}},\;k\in\Nl$.
\end{ex}
Thus we can strike out many gcd-graphs that do not have PST at $\frac{\pi}{2^{k}},$ for all $k\in\Nl$.

\section*{Conclusions}
Quantum networks are based on finite graphs. We consider PST with respect to the adjacency matrix of a graph as the Hamiltonian of a quantum system. We have studied gcd-graphs for PST. Here we find a method to construct gcd-graphs having PST. We show that if an abelian group has an order divisible by $4$ then there exists a connected gcd-graph over that group exhibiting PST. In fact there are many such graphs. In Lemma \ref{n1}, Lemma \ref{n2} and Lemma \ref{n3}, we can observe that there are many other choices for the set $\D\in\mathscr{D}$. In particular, in Lemma \ref{n1}, if we consider a set of divisor tuples $\D$ where each tuples in $\D$ has first component $1$ and the remaining components are any divisors, then also the set $\D\in\mathscr{D}$. Finally, in Theorem \ref{6e}, a necessary and sufficient condition is given for a certain class of gcd-graphs to exhibit periodicity at time $\pi$. From this we find a necessary condition to have PST in a certain class of gcd-graphs at some specific times. This gives a partial characterization of gcd-graphs having perfect state transfer.
\par There are few scopes for further research in this direction.
\begin{itemize}
\item Here we have results which find PST in gcd-graphs at time $\frac{\pi}{2}$. Also we find some gcd-graphs not having PST at $\frac{\pi}{2^{k}}$ for all $k\in\Nl$. One can try to find PST in these gcd-graphs at other possible times.
\item One can also try to find PST on other gcd-graphs that are not covered in this article.
\item Finally, one can try to find PST in integral Cayley graphs which are in fact not gcd-graphs and if possible, characterize all such graphs having perfect state transfer.
\end{itemize}

\section*{Acknowledgment}
We thank the anonymous reviewer(s) for the useful comments in an earlier manuscript.

\end{document}